\documentclass[a4paper,10pt,reqno, english]{amsart}

\usepackage{amsmath,amssymb,amscd,amsthm,amsfonts}
\usepackage{graphicx,subfigure}
\usepackage{hyperref}
\usepackage{dsfont}
\usepackage[nobysame, alphabetic]{amsrefs}
\usepackage{tikz}
\usepackage[capitalise]{cleveref}

\newtheorem{theorem}{Theorem}
\newtheorem{lemma}{Lemma}
\newtheorem{claim}{Claim}
\newtheorem{conjecture}{Conjecture}
\newtheorem{problem}{Problem}
\newtheorem{definition}{Definition}

\newcommand{\Z}{\mathbb{Z}}

\newcommand{\ff}{\mathcal{F}}

\def\zz{\mathds{Z}}
\def\rr{\mathds{R}}
\def\rp{\mathds{R}\mathds{P}}

\DeclareMathOperator{\conv}{conv}

\title{Tverberg cores and Kalai's
cascade conjecture}

\hypersetup{
  pdftitle={Tverberg cores and Kalai's cascade conjecture},
  pdfauthor={Pablo Soberon}
}

\author[Sober\'on]{Pablo Sober\'on}\address{Baruch College, City University of New York, One Bernard Baruch Way, New York, NY 10010, United States} 
\email{psoberon@gc.cuny.edu}

\thanks{
The research of P. Sober\'on is supported by NSF CAREER grant DMS-237324 and a PSC-CUNY Trad B award.}

\keywords{Tverberg theorem, topological Tverberg theorem, Kalai's cascade conjecture, equivariant topology, characteristic classes.}

\subjclass[2020]{Primary 52A37; Secondary 55M20, 55R40, 55R91.}

\begin{document}

\maketitle

\begin{abstract}
We study topological analogues of Kalai's cascade conjecture. Given a continuous map from an \(n\)-simplex to \(\mathbb R^d\), let \(T_r(f)\) be the set of points contained in the images of \(r\) pairwise disjoint faces. We prove that if \(r\) is a prime power and \(\dim T_r(f)\le k\), then there exists a point that remains an \(r\)-Tverberg point after any \(t\) vertices are deleted, provided \(n=(r-1)(d+1)+t(k+1)\).
For \(t=1\), this gives a topological analogue of a standard consequence of Kalai's cascade conjecture. We also confirm the cascade conjecture for finite point sets whose Radon set is \(0\)-dimensional.
\end{abstract}

%\tableofcontents

\section{Introduction}

Tverberg's theorem is a fundamental result in combinatorial geometry concerning the structure of finite sets of points in $\rr^d$.  For sufficiently large finite sets in $\rr^d$, it guarantees the existence of partitions into parts whose convex hulls intersect.  There are numerous generalizations of Tverberg's theorem, and its affine and topological variants remain central topics in combinatorial geometry \cites{Barany2018, DeLoera2019}.  

\begin{theorem}[Tverberg 1966 \cite{Tverberg:1966tb}]
    Let $r,d$ be positive integers.  For any set of $(r-1)(d+1)+1$ points in $\rr^d$ there exists a partition into $r$ parts whose convex hulls intersect.
\end{theorem}

We call a partition as above a Tverberg $r$-partition of the set.  A remarkable feature of this result is that it is not purely affine \cite{Barany1981}.  For prime-power $r$ the same conclusion holds if we replace the set of points and their convex hulls by the images of the faces of a high-dimensional simplex that is mapped continuously to $\rr^d$.

%One  of the most important generalizations of Tverberg's theorem is the topological version.

\begin{theorem}[\"Ozaydin 1987, Volovikov 1996 \cites{Oza87, Volovikov:1996up}]
    Let $r,d$ be positive integers, $n=(r-1)(d+1)$, and $\Delta^n$ be an $n$-dimensional simplex.  Assume $r$ is a prime power.  For any continuous map $f:\Delta^n\to \rr^d$ there exist $r$ pairwise disjoint faces of $\Delta^n$ whose images intersect.  
\end{theorem}

The condition that $r$ is a prime power is necessary, as shown by Frick building on the work of Mabillard and Wagner \cite{Frick:2015wp, Mabillard2014}.  The topological Tverberg theorem and its variations have been a key factor in the development of topological combinatorics \cites{Blagojevic:2017bl, Karasev:2009hq, Zivaljevic2017, Skopenkov2018}.

A natural way to look beyond the existence of a Tverberg partition is to study the set of all possible points of overlap.  For a finite set $S \subset \rr^d$ and a positive integer $r$, denote by $T_r(S)$ the set of \(r\)-Tverberg points of \(S\).  Formally,
\[
T_r(S) = \left\{ p \in \rr^d:
\begin{tabular}{c}
\mbox{There exists a partition of $S$ into $r$ parts}  \\
\mbox{$A_1,\dots, A_r$ with $p \in \conv(A_i)$ for all $i=1,\dots,r$}
\end{tabular}  \right\}.
\]

Kalai's cascade conjecture asks for more than the existence of single points of overlap, it seeks to establish structural properties of all Tverberg partitions.  Kalai conjectured a global trade-off among the dimensions of the successive Tverberg sets of a finite point set. Roughly speaking, if the set of \(r\)-Tverberg points is small for some values of $r$, then additional Tverberg structure should appear elsewhere.  

This is known as Kalai's ``cascade'' conjecture, which is now a long-standing open problem. The results of this paper show that this dimension hypothesis has robust consequences even for continuous maps of simplices.   With the convention $ \dim (\emptyset)=-1$, the cascade conjecture says the following:
%The results of this paper are motivated by Kalai's longstanding ``cascade conjecture'', related to the dimension of the Tverberg points of a set.  Our results show that the conditions for the cascade conjecture indeed have structural implications for the set of Tverberg partitions even in the topological setting.  
%  If we consider $-1 = \dim (\emptyset)$, then Kalai's conjecture says the following:

\begin{conjecture}[Kalai 1974 (see e.g., \cites{Kalai2000, Barany2022})]\label{conj:cascade}
Let $S \subset \rr^d$ be a finite set.  Then, $\sum_{r=1}^{|S|} \dim(T_r(S)) \ge 0.$
\end{conjecture}

The only case that has been solved is $d \le 2$ \cite{KadariMSc} and for sets in sufficiently general position \cite{Reay1979}.  Schnider recently proved a variation of the cascade conjecture if we replace $T_r(S)$ by the set of points at Tukey depth greater than or equal to $r$ \cite{Schnider2023}.  A direct consequence of the cascade conjecture is the following result, which is also unconfirmed except for $r=1$, in which case it is exactly Radon's theorem \cite{Radon:1921vh}, and for $d\le 2$, which follows from Kadari's work.

\begin{conjecture}\label{conj:kalai-simplified}
    Let $d, r$ be positive integers and $k$ a non-negative integer.  If $S$ is a set of $(r-1)(d+1)+k+2$ points in $\rr^d$ and $\dim(T_{r}(S)) \le k$, then $S$ has at least one Tverberg $(r+1)$-partition.
\end{conjecture}

Setting $k=d$ in \cref{conj:kalai-simplified} recovers Tverberg's theorem.

\cref{conj:kalai-simplified} remains open even for small values of \(r\) and \(d\).  The results of this
paper show that the same dimension hypothesis forces a robust common
Tverberg point: after replacing the desired \((r+1)\)-Tverberg partition by the
non-emptiness of the \(r\)-Tverberg core (defined below), the conclusion admits a topological
prime-power analogue.

We define the $t$-th Tverberg $r$-core of a set $S$ as follows.
\[
C^t_r(S) = \bigcap_{S' \in \binom{S}{t}}T_r(S\setminus S').
\]
In other words, $p$ remains a Tverberg point even if we remove any $t$ points of $S$.  A critical case is $t=1$, as our results for $t=1$ are the closest to \cref{conj:kalai-simplified}.  Notice that $T_{r+t}(S) \subset C^t_r(S) \subset T_{r}(S)$.  We can interpret the condition $p \in C^t_r(f)$ as a robust version of $p \in T_r(f)$.  We first extend these definitions to the topological setting.

\begin{definition}
    Let $r,d,n,t$ be positive integers and $\Delta^n$ be an $n$-dimensional simplex.  For a continuous function $f: \Delta^n \to \rr^d$ we define $C^t_r(f)$, the $t$-th Tverberg $r$-core of $f$ as the set of points $p \in \rr^d$ such that for each face $F$ of codimension $t$ of $\Delta^n$, there exist $r$ pairwise disjoint faces of $F$ whose images all contain $p$.
\end{definition}

If $f$ is a linear map, then $ C^t_r(f)=C^t_r(S)$, where $S \subset \rr^d$ is the image of the vertices of $\Delta^n$.  We can also extend the definition of Tverberg points to continuous maps.  Given $f: \Delta^n \to \rr^d$, we define

\[
T_r(f) = \left\{ p \in \rr^d:
\begin{tabular}{c}
\mbox{There exist $r$ pairwise disjoint faces $F_1,\dots,F_r$ of $\Delta^n$}  \\
\mbox{such that $p \in f(F_i)$ for all $i=1,\dots,r$}
\end{tabular}  \right\}.
\]

%The topological Tverberg theorem states that $T_r(f) \neq \emptyset$ for continuous functions $f: \Delta^n \to \rr^d$ with $n=(r-1)(d+1)$ as long as $r$ is a prime power \cites{Oza87, Volovikov:1996up}.  The condition on $r$ is necessary \cite{Frick:2015wp}.  The topological Tverberg theorem and its variations have motivated the development of a wide range of 

We present two main results.  The first is the following topological extension of \cref{conj:kalai-simplified}.  Throughout this manuscript, \(\dim\) denotes Lebesgue covering dimension. Thus, for a
compact metric space \(Y\), the condition \(\dim Y\le k\) means that every
finite open cover of \(Y\) has a finite open refinement of order at most
\(k+1\).

\begin{theorem}\label{thm:main}
    Let $r,d,t$ be positive integers such that $r=q^m$ for some positive integer $m$ and a prime number $q$ and let $k$ be a non-negative integer.  Let $n=(r-1)(d+1)+t(k+1)$, $\Delta^n$ be an $n$-dimensional simplex, and $f:\Delta^n \to \rr^d$ be a continuous map.  If $\dim (T_r(f)) \le k$, then $C^t_r(f) \neq \emptyset$.
\end{theorem}

Our second main result confirms Conjecture 1 when \(\dim T_2(S)=0\).

\begin{theorem}\label{thm:cascade-case}
    Let \(d\) be  a positive integer and $S \subset \rr^d$ be a finite set.  Suppose $\dim \big(T_2(S)\big) \le 0$.  Then, 
    \[\sum_{r=1}^{|S|} \dim(T_r(S)) \ge 0.\]
\end{theorem}

In particular, for $t=1$ and $f$ linear, the parameters of \cref{thm:main} match those of \cref{conj:kalai-simplified}.  We obtain $C^1_r(f) \neq \emptyset$ instead of $T_{r+1}(f) \neq \emptyset$. We first prove \cref{thm:main} for $r=2, t=1$ when $f$ is linear, as the proof is much easier to describe and helps build intuition on the topic.

\cref{thm:main} also holds for $r=1$ if we allow $m=0$.  This particular instance was proved earlier by Karasev \cite{Karasev2012} as a topological extension of Rado's centerpoint theorem \cite{Rado1946}.  We include a short no-retraction proof  in \cref{sec:rado} for completeness
and because it parallels the fiber-cover argument used for the proof of \cref{thm:main}.

The methods we use are topological.  We reduce the problem to showing that some degree-1 cohomology class of an associated bundle does not vanish in the set of Tverberg partitions that overlap in a particular point $p$.  Our approach is similar to Sarkaria's proof of the topological Tverberg theorem using characteristic classes \cite{Sarkaria2000} (see de Longueville's corrected expository notes \cite{Longueville2002}).  Readers familiar with the subject will notice that applying the methods of this paper for $n=(r-1)(d+1)$ gives a full proof of the topological Tverberg theorem.

As mentioned above, the topological Tverberg theorem fails if $r$ is not a prime power.  Therefore, for \cref{thm:main} with $k=d$, it is possible that $T_{r+1}(f)=\emptyset$.  In other words, the topological version of \cref{conj:kalai-simplified} may fail, so some modification of the conclusion is unavoidable in the topological setting.

We prove the case \(r=2\), \(t=1\), for linear maps in \cref{sec:radon}.
In \cref{sec:preliminaries}, we collect the two topological tools needed for
the general proof: a higher-order cohomological vanishing lemma and a
Thom-class lemma.  We prove \cref{thm:main} in \cref{sec:main-proof}.  We prove the case $r=1$ in \cref{sec:rado}.  We prove \cref{thm:confirm-kalai-simple} in \cref{sec:conj-confirmed-case}.
We conclude with remarks in \cref{sec:remarks}.

\section{Radon partitions}\label{sec:radon}

In this section, we present a simple approach that proves \cref{thm:main} for $r=2, t=1$ when the map $f$ is linear.  The advantage of this special case is that the geometric ideas behind the proof are natural, and the argument showcases the moment when topological tools play an essential part.

A Tverberg partition with two parts is commonly known as a Radon partition.  To prove the special case of this section we first need to find a way to parametrize all Radon partitions of a finite set of points in $\rr^d$.

One way to do this is to use a construction similar to the Gale dual of the set of points.  We denote by $n+1$ the number of points in our set, to align with the notation from \cref{thm:main}.  Formally, given a set $S$ of $n+1$ points $x_1,\dots, x_{n+1}$ in $\rr^d$, consider $L$ to be the $(d+1)\times (n+1)$ matrix where $i$-th column is the vector $(x_i,1)$.  The rank of $L$ is at most $d+1$, so $\dim \ker L \ge n-d$.

The space $\ker L$ represents all affine dependences of the points in $S$.  That is, all $(n+1)$-tuples of coefficients $\alpha_1,\dots,\alpha_{n+1}$ such that $\sum_i \alpha_i x_i = 0$ and $\sum_i \alpha_i = 0$.

To distinguish ordered Radon partitions of $S$, we take $X$ to be the $\ell_1$-sphere of $\ker L$ of radius $2$.  Those are the non-zero vectors $\alpha = (\alpha_1,\dots, \alpha_{n+1}) \in \ker L$ for which $\sum_{i}|\alpha_i|=2$.

 A vector $\alpha = (\alpha_1,\dots,\alpha_{n+1}) \in X$ corresponds to a Radon partition, by defining $A = \{x_i: \alpha_i > 0\}$ and $B = \{x_i: \alpha_i < 0\}$.  The absolute values of the coordinates of $\alpha$ give us exactly the coefficients of convex combinations of $A$ and $B$ that share a point.   

 Therefore, for a set of $n+1 = d+3+k$ points the set of ordered Radon partitions with coefficients will be a sphere of dimension greater than or equal to $(k+1)$.  This sphere has a continuous map to the set $T_2(S)$ of Radon points of $S$, which maps $u$ and $-u$ to the same point, for all $u$ in the sphere.  Intuitively, if $\dim (T_2(S)) \le k$, there will be a Radon point in which its preimage contains a path between two antipodal points of the sphere.  We require some additional tools to prove this formally, but it is the essence of the argument.  We use \v{C}ech cohomology with $\zz_2$ coefficients.  We state the following lemma with more general coefficients since it will be useful to compare it with the version we need for higher values of $t$, in which we use $\zz_q$ coefficients for some prime $q$.  We use the convention from the introduction that \(\dim\) denotes Lebesgue
covering dimension.
 
 \begin{figure}
     \centering
     \includegraphics[width=0.8\linewidth]{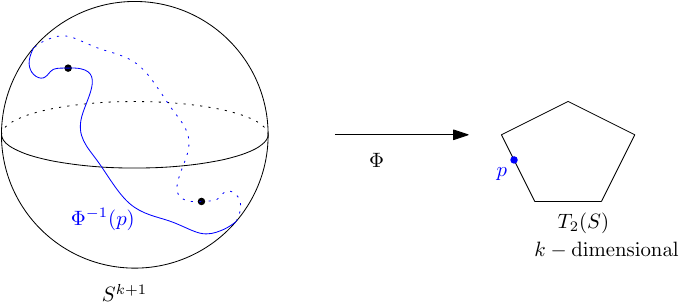}
     \caption{To show that the Radon core is non-empty, we will show the following fact.  If we map a $(k+1)$-dimensional sphere to a $k$-dimensional space using a continuous even map $\Phi$, there will be a point $p$ whose preimage connects two antipodal points of the sphere.}
     \label{fig:sphere}
 \end{figure}

\begin{lemma}\label{claim:important}
Let $k$ be a non-negative integer and $X, Y$ be compact metric spaces and $q$ be a prime number.  Let $f:X\to Y$ be a continuous map and $c \in H^1 (X; \zz_q)$.  If $\dim Y \le k$ and for every $y \in Y$ the restriction of $c$ to the fiber $f^{-1}(y)$ is $0$, then
\[
c^{k+1} =0 \in H^{k+1}(X;\zz_q).
\]
\end{lemma}

%We use \cref{claim:important} with $R=\zz_2$.  We state it in general since we need to use it with $R=\zz_q$ for an odd prime $p$ in \cref{sec:odd-primes}.

\begin{proof}
    Fix $y \in Y$.  We have $f^{-1}(y)$ is compact and $c|_{f^{-1}(y)}=0$.  By continuity of \v{C}ech cohomology over neighborhoods of compact sets, there exists an open set $W_y \subset X$ such that $f^{-1}(y) \subset W_y$ and \(c|_{W_y}=0\).  The set $X \setminus W_y$ is compact and does not intersect $f^{-1}(y)$, so $y \not\in f(X \setminus W_y)$.

    The set $f(X \setminus W_y)$ is compact, and therefore closed in $Y$.  In particular, $U_y = Y \setminus f(X \setminus W_y)$ is an open neighborhood of $y$, and it satisfies $f^{-1}(U_y) \subset W_y$.  In particular, $c|_{f^{-1}(U_y)}=0$.  The sets $\{U_y: y \in Y\}$ form an open cover of $Y$.
    
%    Since $c|_{f^{-1}(y)}=0$, there exists an open neighborhood $U_y \subset Y$ of $y$ such that $c|_{f^{-1}(U_y)} = 0$.

    Since $Y$ is compact, we can find a finite subcover $\mathcal{U}$ of $Y$.  As $\dim Y \le k$, there  is a finite refinement $\mathcal{V} = \{V_{\alpha}: \alpha \in A\}$ of $\mathcal{U}$ of order at most $k+1$ (i.e., any $k+2$ different sets of $\mathcal{V}$ have an empty intersection).  For each \(\alpha\in A\), choose \(y(\alpha)\in Y\) such that
\(V_\alpha\subset U_{y(\alpha)}\).

For $\alpha \in A$, define $F_{\alpha} = f^{-1}(V_{\alpha})$.  Since \(V_\alpha\subset U_{y(\alpha)}\), we have
\(F_\alpha\subset f^{-1}(U_{y(\alpha)})\), which implies
\(c|_{F_\alpha}=0\).  The set $\mathcal{F} = \{F_{\alpha}: \alpha \in A\}$ is an open cover of $X$ of order at most $k+1$ such that $c|_{F_{\alpha}} = 0$ for all $\alpha$.

    Because $c$ restricts trivially to every  $F_{\alpha}$, the class $c$ can be represented by transition data subordinate to $\ff$.  In other words, after choosing all local $0$-cochains/trivialization over each $F_{\alpha}$, the differences on pairwise intersections give a \v{C}ech $1$-cocycle representing $c$.

    Then $c^{k+1}$ is represented by the \v{C}ech cup product of $k+1$ copies of that $1$-cocycle.  This is a \v{C}ech cochain of degree $k+1$ subordinate to $\ff$.  However, no $k+2$ members of $\ff$ have non-empty intersection.  Using normalized \v{C}ech cochains subordinate to \(\mathcal F\), the group
\(\check C^{k+1}(\mathcal F;\zz_q)\) is zero, because such cochains are supported
on \((k+2)\)-fold intersections and no such intersections are non-empty.  This implies $c^{k+1}=0$
\end{proof}

\begin{proof}[Proof of \cref{thm:main} for linear Radon partitions]

If $f$ is linear and $r=2$, studying Radon partitions for $f$ is the same as studying the Radon partitions of the images of the vertices of $\Delta^n$.  Let $S = \{x_1,\dots,x_{d+3+k}\}$ be the set of the $n+1=d+3+k$ images of the vertices of $\Delta^n$.  We construct the space $X$ that parametrizes ordered Radon partitions as described above.  Assume that $\dim T_2(S) \le k$.  By construction $\dim X \ge k+1$.  We may assume without loss of generality that the affine span of $S$ is $\rr^d$, so $X \cong S^{k+1}$.  For each $\alpha \in \rr$, let $\alpha_{+} = \max\{\alpha,0\}$.  

Then for $\alpha = (\alpha_1,\dots, \alpha_{d+3+k}) \in X$ we have a natural projection $\Phi: X \to T_2(S)$ defined by $ \Phi (\alpha) = \sum_{i} (\alpha_i)_+x_i$.

Since $\Phi (\alpha) = \Phi (-\alpha)$, the projection factors through the antipodal double cover $\pi : S^{k+1} \to \rp^{k+1}$ as $\Phi = \bar\Phi \circ \pi$ for a continuous map $\bar \Phi : \rp^{k+1}\to T_2(S)$.

Now consider the cohomology ring of $\rp^{k+1}$ with $\zz_2$ coefficients.  Let $\xi\in H^1(\rp^{k+1};\Z_2)$ be the standard generator. Since
\[
H^*(\rp^{k+1};\Z_2)\cong \Z_2[\xi]/(\xi^{k+2}),
\]
we have $\xi^{k+1}\neq 0 \in H^{k+1}(\rp^{k+1};\zz_2).$

We can now apply \cref{claim:important} to $\bar\Phi$ and $c=\xi$.  Since we know the conclusion of the lemma fails, we must have a $y \in T_2(S)$ such that $\xi|_{\bar\Phi^{-1}(y)}\neq 0$.

Write $\bar \Phi^{-1}(y)$ as a disjoint union of its connected components.  Since \v{C}ech cohomology with $\zz_2$ coefficients is additive on connected components, there exists at least one connected component $C \subset \bar\Phi^{-1}(y)$ for which $\xi|_{C} \neq 0$.

Since $\xi|_{C} \neq 0$, the restricted double cover is nontrivial.  Because $C$ is connected, a nontrivial double cover is also connected.  Since the map $\Phi$ is piecewise linear, connectedness of the preimage of a point implies path-connectedness.  This means that two antipodal Radon partitions in $\Phi^{-1}(y)$ are connected via a path that projects to $y$ via $\Phi$.

Let $(A,B)$ be this partition.  Ultimately, we have to reach $(B,A)$.  As we follow the path connecting $(A,B)$ to $(B,A)$ in $\Phi^{-1}(y)$, let us look at what happens to the partition.  The only possible ways in which we modify the partition is that we either remove points or add points to each set one by one, and throughout this process they remain pairwise disjoint and contain $y$ in their convex hull.  Since every $x_i \in S$ must switch sides, there must be a moment when it is not used in either set of the partition, which means that $y \in T_2(S\setminus \{x_i\})$.  In other words, $y \in C^1_2(S)$.

\end{proof}

\section{Topological preliminaries and notation}\label{sec:preliminaries}

%In this section, we prove the following theorem.

%\begin{theorem}\label{thm:main-even}
 %   Let $r,d,k$ be positive integers such that $r=2^m$ for some positive integer $m$.  Let $n=(r-1)(d+1)+k+1$, $\Delta^n$ be an $n$-dimensional simplex, and $f:\Delta^n \to \rr^d$ be a continuous map.  If $\dim (T_r(f)) \le k$, then $C_r(f) \neq \emptyset$.
%\end{theorem}

Throughout the rest of the manuscript, let
\[
r=q^m,
\qquad
G=(\zz_q)^m,
\qquad
D=(r-1)(d+1),
\qquad
L=t(k+1),
\qquad
n=D+L.
\]
We use \v{C}ech cohomology with \(\zz_q\) coefficients.  We will say that an
\(r\)-tuple \((F_1,\dots,F_r)\) of faces of \(\Delta^n\) heavily covers
\(y\in \rr^d\) if \(y\in f(F_i)\) for all \(i\).

Before proving \cref{thm:main}, we require two topological lemmas.  The proof in \cref{sec:radon} had two important components.  The first was a dimension-reduction argument using the kernel of a linear map.  The second was a topological lemma regarding the vanishing of powers of a degree-$1$ cohomology class.

To establish the analogue of \cref{claim:important} we first introduce the Bockstein homomorphism of $\zz_q$.  We describe the Bockstein homomorphism fully for completeness. 

Consider the standard exact sequence
\[
0 \to \zz_q \underset{\times q}{\to} \zz_{q^2} \underset{\mod q}{\to} \zz_q \to 0.
\]
This induces a map $\beta: H^i (X;\zz_q) \to H^{i+1}(X;\zz_q)$.  We make the degree-$1$ case explicit.  Given $c \in H^1(X;\zz_q)$, choose a $1$-cocycle representative $c\in C^1(X;\zz_q)$, which we can lift to a $1$-cochain $\tilde{c}\in C^1(X;\zz_{q^2})$.  Since $c$ is a cocycle mod $q$, $\delta \tilde{c}$ is a multiple of $q$, which means $\delta \tilde{c} = q b$ for a $\zz_q$-valued $2$-cochain $b$.  Now let $\beta(c) := [b] \in H^2(X;\zz_q)$ be the Bockstein of $c$.

\begin{lemma}\label{lem:higher-fiber-cover}
Let \(t\) be a positive integer, \(k\) be a non-negative integer, and \(X,Y\) be compact metric spaces.
Let \(\rho:X\to Y\) be a continuous map and suppose that \(\dim Y\le k\).
Let \(c\in H^1(X;\zz_q)\) and $\beta(c) \in H^2(X;\zz_q)$ be its Bockstein.

Assume that there is a finite family of open subsets \(\mathcal O=\{O_i\}_{i\in I}\)
of \(X\) such that \( c|_{O_i}=0 \qquad\text{for all }i\in I\), and such that for every \(y\in Y\), there is a subset \(I_y\subset I\) with \(|I_y|\le t\) and \(\rho^{-1}(y)\subset \bigcup_{i\in I_y}O_i\).

Let
\[
B_{t(k+1)}(c)=
\begin{cases}
c\,\beta(c)^{(t(k+1)-1)/2}, & \text{if }t(k+1)\text{ is odd},\\[4pt]
\beta(c)^{t(k+1)/2}, & \text{if }t(k+1)\text{ is even}.
\end{cases}
\]
Then \(B_{t(k+1)}(c)=0\in H^{t(k+1)}(X;\zz_q)\).
\end{lemma}

For $q=2$, for $c \in H^1(X;\zz_2)$ we have $\beta(c) = c^2$.  This shows for $t=1,q=2$ we recover $B_{t(k+1)}=c^{k+1}$, so the result above generalizes \cref{claim:important}.

\begin{proof}
Recall $L=t(k+1)$.  Fix \(y\in Y\).  By assumption, there is a set \(I_y\subset I\) with \(|I_y|\le t\) such that \(\rho^{-1}(y)\subset \bigcup_{i\in I_y}O_i\).

For a fixed $y$, let $W_y = \bigcup_{i \in I_y}O_i$.  The set $X \setminus W_y$ is compact, and since $\rho^{-1}(y) \subset W_y$ we have $y \not\in \rho (X \setminus W_y)$.  Because $Y$ is metric and $\rho(X \setminus W_y)$ is compact, it is closed.  Therefore, $U_y = Y \setminus \rho(X \setminus W_y)$ is an open neighborhood of $y$.  Moreover, \(\rho^{-1}(U_y)\subset \bigcup_{i\in I_y}O_i\).

%Indeed, otherwise one could find a sequence \(y_j\to y\) and points
%\[
%x_j\in \rho^{-1}(y_j)\setminus \bigcup_{i\in I_y}O_i.
%\]
%Passing to a convergent subsequence, compactness gives a limit point
%\(x\in \rho^{-1}(y)\), contradicting the inclusion above.

The sets \(U_y\) cover \(Y\).  By compactness, we can choose a finite subcover, and then refine it by a finite open cover \(\mathcal V=\{V_\alpha\}_{\alpha\in A}\)
of order at most \(k+1\).  For each \(\alpha\), choose \(y(\alpha)\in Y\) such that \(V_\alpha\subset U_{y(\alpha)}\).  Let \(I_\alpha=I_{y(\alpha)}\).

For \(\alpha\in A\) and \(i\in I_\alpha\), define
\[
W_{\alpha,i}=\rho^{-1}(V_\alpha)\cap O_i.
\]
The family \( \mathcal W=\{W_{\alpha,i}:\alpha\in A,\ i\in I_\alpha\}\) is an open cover of \(X\).  Moreover, every member of \(\mathcal W\) is
contained in some \(O_i\), and therefore \(c|_{W_{\alpha,i}}=0\).

The order of \(\mathcal W\) is at most \(L\).  Indeed, if \(x\in X\), then
\(\rho(x)\) lies in at most \(k+1\) of the sets \(V_\alpha\), and for each such
\(\alpha\) there are at most \(t\) choices of \(i\in I_\alpha\).

Since $c|_{W_{\alpha,i}}=0$ for every member of $\mathcal{W}$, we can choose local trivializations (equivalently local $0$-cochains) over the members $W_{\alpha,i}$ of $\mathcal{W}$.  Their differences on pairwise intersections give a normalized \v{C}ech $1$-cocycle $a \in \check{C}^1(\mathcal{W};\zz_q)$ representing $c$.

Now let us show that we can also represent $\beta(c)$ on this cover.  This follows from the fact that the Bockstein homomorphism is natural, but we do it explicitly.  Choose lifts of $a$ to a normalized \v{C}ech $1$-cochain $\tilde{a} \in \check{C}^1(\mathcal{W};\zz_{q^2}).$  Since $a$ is a cocycle modulo $q$, the coboundary $\delta \tilde{a}$ is divisible by $q$.  We can therefore write $\delta \tilde{a} = qb$ for some normalized \v{C}ech $2$-cochain $b \in \check{C}^2(\mathcal{W};\zz_q)$.  Then $b$ is a \v{C}ech $2$-cocycle representing $\beta(c)$.

Using $a$ for $c$ and $b$ for $\beta(c)$, the class $B_L(c)$ is represented subordinate to $\mathcal{W}$ by taking 
\[
\begin{cases}
a\smile b^{(L-1)/2}, & \text{if }L\text{ is odd},\\[4pt]
b^{L/2}, & \text{if }L\text{ is even}.
\end{cases}
\]

Since $\mathcal{W}$ has order at most $L$, no $L+1$ members of $\mathcal{W}$ have a non-empty intersection.  Therefore, $\check{C}^{L}(\mathcal{W};\zz_q) = 0$ for normalized \v{C}ech cochains.  The degree-$L$ representative of $B_L(c)$ is zero and therefore
\[
B_L(c) = 0 \in H^L(X;\zz_q).
\]
\end{proof}

To replace the dimension-reduction argument from \cref{sec:radon}, we will use some characteristic classes of vector bundles.  The standard characteristic classes to detect zeros of vector bundles are the Stiefel--Whitney classes when we work with $\zz_2$ coefficients, and the Euler class when we work with $\zz_q$ coefficients for a prime $q$ (provided that the bundle is oriented).

For the bundles that we construct in \cref{sec:main-proof}, we have explicit descriptions of the relevant characteristic classes.  However, we state the following lemma in a more general setting.  The lemma is a standard consequence of the Thom-class formalism; compare with Stasheff--Milnor, where similar statements are discussed for the Euler class of an oriented bundle \cite{Milnor74}*{Chpt. 9}.

\begin{lemma}\label{lem:thom-cupzero}
Let \(\eta\to X\) be a real vector bundle of rank \(D\), and let
\(\sigma:X\to \eta\) be a section with zero set \(Z(\sigma)\).

For every \(a\in H^\ell(X;\zz_2)\) whose restriction to
\(Z(\sigma)\) is zero,
\[
a\smile w_D(\eta)=0\in H^{\ell+D}(X;\zz_2),
\]
where $w_D(\eta)$ is the top Stiefel--Whitney class of $\eta$.

If \(q\) is an odd prime and \(\eta\) is oriented, then for every
\(a\in H^\ell(X;\zz_q)\) whose restriction to \(Z(\sigma)\) is zero,
\[
a\smile e(\eta)=0\in H^{\ell+D}(X;\zz_q),
\]
where $e(\eta)$ is the Euler class of $\eta$.
\end{lemma}

\begin{proof}
We give the proof in the oriented case first.  Let \(u_\eta\in H^D(\eta,\eta\setminus X;\zz_q)\) be the Thom class of \(\eta\).  Since \(\sigma\) is nowhere zero on
\(X\setminus Z(\sigma)\), it induces a map of pairs
\[
\sigma:(X,X\setminus Z(\sigma))\to (\eta,\eta\setminus X).
\]
Hence we obtain a relative class
\[
\tau_\sigma:=\sigma^*(u_\eta)
\in H^D(X,X\setminus Z(\sigma);\zz_q).
\]
Its image under the natural map
\[
H^D(X,X\setminus Z(\sigma);\zz_q)\to H^D(X;\zz_q)
\]
is the Euler class \(e(\eta)\).

Now assume that \(a|_{Z(\sigma)}=0\).  By exactness of the long exact sequence
of the pair \((X,Z(\sigma))\), there exists
\(\tilde a\in H^\ell(X,Z(\sigma);\zz_q)\) whose image in \(H^\ell(X;\zz_q)\) is \(a\).  The relative cup product gives
\[
\tilde a\smile \tau_\sigma
\in
H^{\ell+D}\bigl(X,Z(\sigma)\cup (X\setminus Z(\sigma));\zz_q\bigr)
=
H^{\ell+D}(X,X;\zz_q)=0.
\]
Passing to absolute cohomology, we obtain \(a\smile e(\eta)=0\).

The proof with \(\zz_2\) coefficients
and the top Stiefel--Whitney class is identical; the image of $\tau_{\sigma}$ under the corresponding natural map is the top Stiefel--Whitney class
\end{proof}

\section{Proof of \cref{thm:main}}\label{sec:main-proof}

\begin{proof}[Proof of \cref{thm:main}]
We identify the $r$ parts in an ordered $r$-partition of the vertices of $\Delta^n$ with the elements of $G$.  Denote by $W_G$ the $(r-1)$-dimensional space
\[
W_G = \left\{(a_g)_{g \in G} \in \rr^{|G|}: \sum_{g \in G}a_g = 0\right\}.
\]

To parametrize partitions, we will use the $(n+1)$-fold join of $G$; namely the free $G$-space

\[
K = G^{*(n+1)}= \left\{\sum_{i=1}^{n+1}t_ig_i: g_i \in G, \quad \sum^{n+1}_i t_i = 1, \quad t_i \ge 0 \mbox{ for }i=1,\dots,n+1\right\}.
\]

For $g \in G$, denote $I_g = \{i : g_i = g\}$ and $m_g = \sum_{i \in I_g} t_i$.  Let $x_1,\dots,x_{n+1}$ be the vertices of $\Delta^n$.  For $g \in G$, consider the point $y_g \in \rr^{d+1}$

\[
y_g = \begin{cases}
    m_g \left(f\left(\frac{1}{m_g}\sum_{i \in I_g} t_i x_i \right),1\right) & \mbox{ if }m_g \neq 0\\
    0 & \mbox{ otherwise.}
\end{cases}
\]

Now consider $\{v_g: g \in G\}$ a set of $r$ vectors in $W_G$ forming the vertices of a regular simplex centered at the origin.  The only non-trivial affine dependence of the $v_g$, up to scalar multiplication, is $\sum_{g\in G}v_g = 0$.  Then, we can define the test map

\begin{align*}
    \Psi : K  & \to W_G^{\oplus (d+1)} \\
    \sum_{i=1}^{n+1}t_i g_i & \mapsto \sum_{g \in G} y_g \otimes v_g.
\end{align*}

Note that $\Psi$ is $G$-equivariant and continuous.  A zero of $\Psi$ corresponds to an element of $K$ for which all the points $y_g$ are equal.  The last coordinate of each $y_g$ implies that $m_g = 1/r$ for all $g$.  Therefore, the first $d$ coordinates of each $y_g$ imply that we have $r$ points from pairwise-disjoint faces of $\Delta^n$ that $f$ maps to the same point in $\rr^d$.

In other words, $Z = \Psi^{-1}(0)$ is precisely the space of ordered topological Tverberg $r$-partitions for $f$ with coefficients.  The common overlap point of such a Tverberg partition defines a continuous map
\[
\rho : Z / G \to T_r(f).
\]

Since $Z/G$ is compact and $\rho$ is surjective, the set $T_r(f)$ is also compact.  Let $M = K/G$ and $\xi = K \times_G W_G^{\oplus (d+1)}$ and consider the vector bundle $\xi \to M$.  The $G$-equivariant map $\Psi: K \to W_G^{\oplus (d+1)}$ induces a section 
\begin{align*}
s:M &\to \xi \\
[x] & \mapsto [x,\Psi(x)]
\end{align*}

whose zero set is precisely $Z/G$.  For $\beta$ the Bockstein homomorphism associated with the standard exact sequence 
\[
0 \to \zz_q \underset{\times q}{\to} \zz_{q^2} \underset{\mod q}{\to} \zz_q \to 0
\]
and a non-zero $u \in H^1(BG;\zz_q)$, we define $B_{t(k+1)}(u)=B_L(u)$ as in \cref{lem:higher-fiber-cover},
\[
B_{t(k+1)}(c)=
\begin{cases}
c\,\beta(c)^{(t(k+1)-1)/2}, & \text{if }t(k+1)\text{ is odd},\\[4pt]
\beta(c)^{t(k+1)/2}, & \text{if }t(k+1)\text{ is even}.
\end{cases}
\]

\begin{claim}\label{claim:non-zero-class}
    There is a class $c \in H^1(Z/G;\zz_q)$ such that $B_L(c) \neq 0.$
\end{claim}

\begin{proof}
    We break the proof into two cases.

    \textbf{Case 1, $q=2$.}  In this case, $B_L(u) = u^L$.  For $G=(\zz_2)^m$, we have a full description of $H^*(BG;\zz_2)$.  Namely,
    \[
H^{*}(BG;\zz_2) = \zz_2[t_1,\dots,t_m] \quad \deg t_i = 1.
\]
The formula for the top Stiefel--Whitney class for $\xi$ is also known \cites{Basu2024, Blagojevic2015}.  More precisely,
\[
w_D(\xi)=w_{(r-1)(d+1)}(\xi) = \left(\prod_{0\neq \alpha \in (\mathds{F}_2)^m}(\alpha_1 t_1 + \dots + \alpha_m t_m) \right)^{d+1}.
\]
Since \(H^*(BG;\zz_2)\) is a polynomial ring, we have
\[
B_L(u)w_D(\xi) = u^Lw_D(\xi)\neq 0\in H^{D+L}(BG;\zz_2)=H^n(BG;\zz_2).
\]
The free \(G\)-space \(K=G^{*(n+1)}\) is \((n-1)\)-connected, so the classifying
map \(\kappa:M\to BG\) induces an isomorphism in cohomology up to degree \(n-1\) and an injection in
degree \(n\).  Therefore the pullback of \(B_L(u)w_D(\xi)\) to \(M\) is nonzero.
Applying \cref{lem:thom-cupzero} to the section \(s\) of \(\xi\), whose zero set
is \(Z/G\), we conclude that \(B_L(u)|_{Z/G}\neq 0\), so it suffices taking \(c=u|_{Z/G}\) to finish. 

\textbf{Case 2, $q$ is odd.}  For \(G = (\zz_q)^m\) for odd $q$ we also have a full description of $H^1(BG;\zz_q)$.  In this case
\[
H^*(BG;\zz_q)\cong
\Lambda(e_1,\dots,e_m)\otimes \zz_q[t_1,\dots,t_m],
\qquad
\deg e_i=1,\quad \deg t_i=2.
\]
Moreover, the degree-two generators can be chosen so that $t_i = \beta(e_i)$ for all $i$.  In this case the bundle $\xi$ is oriented and we have an explicit expression for the Euler class \cite{Basu2024, Blagojevic2015}.
Up to a nonzero scalar, the Euler class of the corresponding universal bundle
is
\[
e(\xi)=
\prod_{\alpha\in A}(\alpha_1t_1+\cdots+\alpha_m t_m)^{d+1},
\]
where \(A\) is any set of representatives of the nonzero elements of
\((\zz_q)^m\) modulo the involution \(\alpha\sim -\alpha\).

In particular, $e(\xi)$ is a non-zero polynomial of degree $D/2$ in $t_1,\dots, t_m$.  The expression $B_L(u)e(\xi)$ is then a non-zero polynomial in $t_1,\dots, t_m$ if $L$ is even, or a non-zero polynomial in $t_1,\dots,t_m$ multiplied by a non-zero linear term in $e_1,\dots, e_m$ if $L$ is odd.

In either scenario, the description of $H^*(BG;\zz_q)$ shows that 
\[
B_L(u)e(\xi) \neq 0\in H^{D+L}(BG;\zz_q)=H^n(BG;\zz_q).
\]
From this point we can conclude in an identical way as in the first case. The \((n-1)\)-connectedness of \(K=G^{*(n+1)}\) implies that the pullback of \(B_L(u)e(\xi)\) to \(M\) is nonzero.
We can again apply \cref{lem:thom-cupzero} to the section \(s\) of \(\xi\) and conclude that \(B_L(u)|_{Z/G}\neq 0\), and we finish by considering \(c=u|_{Z/G}\).
\end{proof}

%Ultimately, we want to find an element of $T_r(f)$ with a large pre-image under $\rho$.   In the Radon case, we were able to reduce the dimension of the space we needed to analyze by taking the kernel of a linear function.  In the case for this section, where $r=2^m$ and we no longer have the luxury of dealing with linear functions, we lack such a direct reduction in dimension.  As a replacement, we will need an auxiliary lemma.

Now let $\pi: Z \to Z/G$ be the principal $G$-cover.  Since $H^1(BG;\zz_q) \cong \hom (G, \zz_q)$, a non-zero class $u \in H^1(BG;\zz_q)$ corresponds to a non-zero homomorphism $G \to \zz_q$.  %Let $N=\ker u$.

The associated $q$-fold cover is $P_u: Z/\ker(u) \to Z/G$.  Its cohomology class is $c=u|_{Z/G}$.

For each vertex $x_i$ of \(\Delta^n\), define $O_i \subset Z/G$ to be the set of ordered Tverberg partitions in which the coefficient of $x_i$ is positive.  This condition is invariant under the action of $G$, so $O_i$ is a well-defined open subset of $Z/G$.  We claim that $c|_{O_i}=0$ for all $i$.

Over $O_i$, the label of $x_i$ is well-defined.  Taking this label modulo $\ker(u)$ gives a continuous $(G/\ker(u))$-equivariant
\[
P_u|_{O_i}\to (G/\ker(u))\cong \zz_q.
\]
The preimage of $0 \in G/\ker(u)$ gives a continuous section of $P_u|_{O_i}\to O_i$.  Therefore the associated cover is trivial over $O_i$, and therefore $c|_{O_i}=0$.

If we choose the class $c$ from \cref{claim:non-zero-class}, we know that the conclusion of \cref{lem:higher-fiber-cover} is not satisfied.  Therefore, it must be the case that there exists a point $p \in T_r(f)$ such that for every subset $A \subset [n+1]$ of size $t$, $\rho^{-1}(p)$ is not contained in $\bigcup_{i \in A} O_i$.  In other words, there must exist $z_A \in \rho^{-1}(p) \setminus \left( \bigcup_{i \in A} O_i \right)$.

The point $z_A$ corresponds to a Tverberg partition in which $p$ is heavily covered but the vertices $x_i$ with $i \in A$ have coefficient $0$.  Since this holds for any set of $t$ vertices, we have that $p$ is in $C_r^t(f)$, as we wanted to show.

\end{proof}

\section{The case $r=1$}\label{sec:rado}

Let us prove the $r=1$ analogue of \cref{thm:main}.  Note that $T_1(f)$ is just the image of $f$.  A point $p$ is in $C^t_1(f)$ if it is in the image of every face of $\Delta^n$ of co-dimension $t$.  As mentioned in the introduction, this case was proven by Karasev, and we include a new proof that aligns with the techniques from previous sections.

\begin{theorem}[Karasev 2012 \cite{Karasev2012}]\label{thm:topo-rado}
    Let $d,t$ be positive integers and $k$ a non-negative integer.  Let $n=t(k+1)$, $\Delta^n$ be an $n$-dimensional simplex, and $f:\Delta^n \to \rr^d$ be a continuous map.  If $\dim (T_1(f))= \dim f(\Delta^n) \le k$, then $C^t_1(f) \neq \emptyset$.
\end{theorem}

\begin{proof}
Let \(V=\{x_1,\dots,x_{n+1}\}\) be the vertex set of \(\Delta^n\), and consider \(Y=f(\Delta^n)\).
By assumption, \(\dim Y\le k\).  Let \(\lambda_i:\Delta^n\to \rr\) denote the
barycentric coordinate corresponding to the vertex \(x_i\), and define the set \(O_i=\{z\in \Delta^n:\lambda_i(z)>0\}\), which is simply the complement of the facet opposite to vertex $i$.

Suppose, for contradiction, that \(C_1^t(f)=\emptyset\).
Then for every \(p\in Y\), there is a \(t\)-element set of vertices \(A_p\subset V\) such that \(p\) is not in the image under \(f\) of the face spanned by
\(V\setminus A_p\).  Equivalently, \(f^{-1}(p)\subset \bigcup_{x_i\in A_p}O_i\).

Since \(f^{-1}(p)\) is compact and \(\bigcup_{x_i\in A_p}O_i\) is open, there is
an open neighborhood \(U_p\subset Y\) of \(p\) such that \(f^{-1}(U_p)\subset \bigcup_{x_i\in A_p}O_i\).
The sets \(U_p\) cover \(Y\). Since \(\dim Y\le k\), after passing to a finite
refinement, we may choose a finite open cover
\[
\mathcal V=\{V_\alpha\}_{\alpha\in I}
\]
of \(Y\) of order at most \(k+1\), such that each \(V_\alpha\) is contained in
some \(U_{p(\alpha)}\).

For each \(\alpha\in I\) and each \(x_i\in A_{p(\alpha)}\), define
\[
W_{\alpha,i}=f^{-1}(V_\alpha)\cap O_i.
\]
The family
\(\mathcal W=\{W_{\alpha,i}:\alpha\in I,\ x_i\in A_{p(\alpha)}\}\) is an open cover of \(\Delta^n\).

The order of \(\mathcal W\) is at most \(t(k+1)=n\).

Choose a partition of unity \(\{\varphi_{\alpha,i}\}\)
subordinate to the cover \(\mathcal W\).  Consider the continuous map
\begin{align*}
h:\Delta^n &\to \Delta^n \\
z & \mapsto \sum_{\alpha,i}\varphi_{\alpha,i}(z)x_i.
\end{align*}

Since \(W_{\alpha,i}\subset O_i\), if \(\lambda_i(z)=0\), then
\(z\notin O_i\), and in particular \(z\not\in W_{\alpha,i}\).
Therefore the coefficient of \(x_i\) in \(h(z)\) is zero whenever the coefficient of
\(x_i\) in \(z\) is zero.  Consequently, \(h\) maps every face of \(\Delta^n\) to itself.  This means that $h$ is of degree $1$ on the boundary and therefore surjective.

On the other hand, since the order of \(\mathcal W\) is at most \(n\), at every
point \(z\in\Delta^n\) at most \(n\) of the functions \(\varphi_{\alpha,i}(z)\)
are nonzero.  Therefore \(h(z)\) is always a convex combination of at most \(n\)
vertices of \(\Delta^n\).  Since \(\Delta^n\) has \(n+1\) vertices, this implies \(h(\Delta^n)\subset \partial\Delta^n\), a contradiction.
\end{proof}

\section{Proof of \cref{thm:cascade-case}}\label{sec:conj-confirmed-case}

To prove \cref{thm:cascade-case}, we first prove the following theorem.  The case $t=1$ of the theorem below is the case $r=2, k=0$ of \cref{conj:kalai-simplified}.

\begin{theorem}\label{thm:confirm-kalai-simple}
     Let $d$ be a positive integer and $t$ be a non-negative integer.  If $S$ is a set of $d+t+2$ points in $\rr^d$ and $\dim(T_{2}(S)) \le 0$, then $S$ has at least one Tverberg $(2+t)$-partition.
\end{theorem}

\begin{proof}
The case $t=0$ is Radon's theorem, so we assume $t>0$.  We repeat some arguments from \cref{sec:radon}. Since the number of points is greater than $d+1$, we know $T_2(S) \neq \emptyset$, so we have $\dim T_2(S) = 0$. The first step of the proof will be to show that the set of Radon points is connected, and therefore a single point.  We may assume without loss of generality that \(S\) affinely spans \(\mathbb R^d\).  Let $N = d+t+2$.

Let
\[
W=
\left\{
\alpha=(\alpha_1,\ldots,\alpha_N)\in\mathbb R^N:
\sum_{i=1}^N \alpha_i=0,\quad
\sum_{i=1}^N \alpha_i x_i=0
\right\}
\]
be the space of affine dependences among the points of \(S\). Since \(S\) affinely spans \(\mathbb R^d\), the augmented vectors \((x_i,1)\in\mathbb R^{d+1}\) have rank \(d+1\). In particular, \(\dim W=N-(d+1)\ge t+1\).

Consider the \(\ell^1\)-unit sphere in \(W\), normalized by
\[
X=
\left\{
\alpha\in W:
\sum_{i=1}^N |\alpha_i|=2
\right\}.
\]
Since \(\dim W\ge 2\), the space \(X\) is connected.

For \(\alpha\in X\), define \(P(\alpha)=\{i:\alpha_i>0\}\)
 and \(Q(\alpha)=\{i:\alpha_i<0\}\).  This is a Radon partition.  We also have a continuous function $\alpha \mapsto \sum_{i=1}^N \alpha_i^+ x_i$, where \(\alpha_i^+=\max\{\alpha_i,0\}\).  The image of this function is $T_2(S)$.

%define \(\alpha_i^+=\max\{\alpha_i,0\}\). Since \(\sum_i\alpha_i=0\) and \(\sum_i|\alpha_i|=2\), we have \(\sum_{i=1}^N \alpha_i^+=1\).
%Define
%\[
%\Phi:X\to \mathbb R^d,\qquad
%\Phi(\alpha)=\sum_{i=1}^N \alpha_i^+ x_i.
%\]
%If
%\[
%P(\alpha)=\{i:\alpha_i>0\},
%\qquad
%Q(\alpha)=\{i:\alpha_i<0\},
%\]
%then
%\[
%\Phi(\alpha)
%=
%\sum_{i\in P(\alpha)}\alpha_i x_i
%=
%\sum_{i\in Q(\alpha)}(-\alpha_i)x_i.
%\]
%Therefore \(\Phi(\alpha)\) is the Radon point corresponding to the Radon partition
%\[
%P(\alpha),Q(\alpha).
%\]
Every Radon point of \(S\) arises in this way from some normalized affine dependence. This implies that $T_2(S)$ is connected.  Since it has dimension \(0\), it consists of a single point.  We assume without loss of generality that \(T_2(S)=\{0\}\).

%The space \(X\) is connected, so \(T_2(S)\) is connected. Since \(\dim T_2(S)\le 0\), the compact connected set \(T_2(S)\) consists of a single point. Let
%\[
%T_2(S)=\{p\}.
%\]
%Translating \(S\) by \(-p\), we may assume from now on that
%\[
%T_2(S)=\{0\}.
%\]

Every affine dependence among the points of \(S\) has Radon point \(0\). Equivalently, if \(\alpha\in W\), then \(\sum_{\alpha_i>0}\alpha_i x_i=0\)
and \(\sum_{\alpha_i<0}(-\alpha_i)x_i=0\).

For each \(i\), let
\[
\lambda_i:W\to\mathbb R,\qquad
\lambda_i(\alpha)=\alpha_i
\]
be the \(i\)-th coordinate functional. Ignore indices \(i\) for which \(\lambda_i\equiv 0\); such vertices never appear in any affine dependence and may be assigned to any part at the end.

Partition the remaining indices into blocks \(B_1,\ldots,B_s\)
according to proportionality of the nonzero coordinate functionals.  Two indices \(i,j\) are in the same block if and only if there is a nonzero constant \(c\) such that \(\alpha_i=c\alpha_j\) for every affine dependence \(\alpha\in W\).  A block is a maximal set of indices which are locked together in fixed ratios across all affine dependences.  

For each block \(B_j\), there is a nonzero linear functional \(\varphi_j\in W^*\)
and nonzero constants \(c_i\) such that
\(
\lambda_i=c_i\varphi_j\)
{for all } \(i\in B_j\).

\begin{claim}\label{claim:blocks}
    For each block $B_j$ we have \(0 \in \conv \{x_i: i \in B_j\}\).
\end{claim}

We first prove the claim above.  Fix a block \(B_j\). Let
\(H_j=\ker \varphi_j\).
Choose a point \(\beta\in H_j\) which lies in no other coordinate hyperplane \(\ker\lambda_i\) with \(i\notin B_j\). This is possible because the other coordinate hyperplanes do not contain \(H_j\).

Take two nearby points \(\alpha^+,\alpha^-\in W\) on opposite sides of \(H_j\), close enough so that the signs of all coordinates outside \(B_j\) are the same for \(\alpha^+\) and \(\alpha^-\) (in other words, two neighboring cells of the hyperplane arrangement generated by the kernels). Let \(E\) be the common set of positive indices outside \(B_j\).

On the side where \(\varphi_j>0\), the positive indices inside \(B_j\) are precisely those with \(c_i>0\). Since every Radon point is \(0\), we have the linear identity
\[
\sum_{i\in E}\lambda_i(\alpha)x_i
+
\varphi_j(\alpha)\sum_{\substack{i\in B_j\\ c_i>0}}c_i x_i
=
0.
\]
On the side where \(\varphi_j<0\), the positive indices inside \(B_j\) are precisely those with \(c_i<0\). Again using that every Radon point is \(0\), we obtain
\[
\sum_{i\in E}\lambda_i(\alpha)x_i
+
\varphi_j(\alpha)\sum_{\substack{i\in B_j\\ c_i<0}}c_i x_i
=
0.
\]
Both identities hold on open subsets of \(W\), hence hold as linear identities on all of \(W\). Subtracting them gives
\[
\varphi_j(\alpha)
\sum_{i\in B_j}|c_i|x_i
=
0
\]
for all \(\alpha\in W\). Since \(\varphi_j\neq 0\), it follows that \(\sum_{i\in B_j}|c_i|x_i=0\).  All coefficients \(|c_i|\) are positive. Therefore \(0\in \operatorname{conv}\{x_i:i\in B_j\}\) as we wanted to show.

If \(s\), the number of blocks, is greater than or equal to $t+2$, then the disjoint non-empty sets
\[
B_1,\quad B_2, \quad \dots, \quad B_{t+2}
\]
already have convex hulls containing \(0\). They form a Tverberg \((t+2)\)-partition, after assigning any ignored or unused vertices arbitrarily to one of the parts.

It remains to handle the case \(s\le t+1\). The coordinate functionals \(\lambda_i\) span \(W^*\), since the only vector in \(W\) annihilated by all coordinates is the zero vector. Therefore the lines spanned by \(\varphi_1,\ldots,\varphi_s\) span \(W^*\). Since \(\dim W\ge t+1\), we cannot have \(s=t\). This implies that \(s=t+1\)
and \(\dim W=t+1\).
In particular, \(\varphi_1,\varphi_2, \dots, \varphi_{t+1}\) are linearly independent.

For \(j=1,\dots,t+1\), write
\(
\lambda_i=c_i\varphi_j\) for \(i\in B_j\).  Since every \(\alpha\in W\) satisfies \(\sum_{i=1}^N\alpha_i=0\),
we get
\[
\varphi_1(\alpha)\sum_{i\in B_1}c_i
+ \dots + 
\varphi_{t+1}(\alpha)\sum_{i\in B_{t+1}}c_i
=
0
\]
for all \(\alpha\in W\). Since \(\varphi_1,\dots,\varphi_{t+1}\) are independent,
\(
\sum_{i\in B_j}c_i=0\) for \(j\in[t+1]\).
Similarly, from
\(
\sum_{i=1}^N\alpha_i x_i=0\)
for all \(\alpha\in W\), we obtain
\(
\sum_{i\in B_j}c_i x_i=0\)
 {for }\(j\in[t+1]\).

Now consider \(B_1\). Split it into the two sets
\[
B_1^+=\{i\in B_1:c_i>0\},
\qquad
B_1^-=\{i\in B_1:c_i<0\}.
\]
Both sets are non-empty because \(\sum_{i\in B_1}c_i=0 \)
and all \(c_i\neq 0\).

We have two equations for $B_1$.  First, \(\sum_{i\in B_1}c_i x_i=0\) and second, as we proved before, \(\sum_{i\in B_1}|c_i|x_i=0\).  This implies that

\[
\sum_{i\in B_1^+}c_i x_i = \sum_{i\in B_1^-}(-c_i)x_i = 0.
\]
In particular, we have that $0$ is in both $\conv B_1^+$ and \(\conv B_1^-\).  By \cref{claim:blocks} we also have $0 \in \conv \{x_i: i \in B_2\}$.
%\[
%0\in \operatorname{conv}\{x_i:i\in B_1^+\}
%\]
%and
%\[
%0\in \operatorname{conv}\{x_i:i\in B_1^-\}.
%\]
%We also know that
%\[
%0\in \operatorname{conv}\{x_i:i\in B_2\}.
%\]
We obtain a Tverberg $(t+2)$-partition induced by the pairwise disjoint non-empty sets
\[
B_1^+,\qquad B_1^-\qquad, B_2 \quad ,\dots, \quad B_{t+1}.
\]
\end{proof}

Notice that in the second part of the proof, when there are only $t+1$ blocks, we actually find a Tverberg $2(t+1)$-partition with the sets \(B_1^+, B_1^-,\dots,B_{t+1}^+, B_{t+1}^-\) by applying the same arguments to $B_2,\dots, B_{t+1}$.

\begin{proof}[Proof of \cref{thm:cascade-case}]
Let $a$ be the dimension of the affine span of $S$.  If $|S| \le a+1$, then \(\sum_{r=1}^{|S|} \dim(T_r(S)) \ge a + a(-1)\ge 0\).  Since $|S| \ge a+2$, by Radon's theorem there exists at least one Radon partition.   \cref{thm:confirm-kalai-simple} implies we have Tverberg partitions with at least $|S|-a$ parts.  The dimension of $T_1(S)$ is $a$, and in the sum $\sum_{r=1}^{|S|} \dim(T_r(S))$ there are at most $a$ negative terms, so $\sum_{r=1}^{|S|} \dim(T_r(S)) \ge a + a(-1) \ge 0$.    
\end{proof}

\section{Remarks}\label{sec:remarks}

As mentioned in the introduction, if we apply \cref{thm:main} with $k=d$, we obtain a non-empty set $C^1_r(f)$.  As long as $r+1$ is not a prime power, it is possible that $T_{r+1}(f)=\emptyset$.  In other words, since the topological Tverberg theorem fails, modifications to the conclusion are unavoidable.

Unfortunately, our results do not imply the full cascade conjecture for linear maps.  For the case $r=2, t=1$, the proof of \cref{thm:main} gives us slightly more structure on the Radon partitions than just $C^1_2(S) \neq \emptyset$.  It shows that there exists a Radon partition $(A,B)$ such that by iteratively
\begin{itemize}
    \item adding points of $S \setminus (A \cup B)$ to either $A$ or $B$ or
    \item removing points from $A$ or $B$,
\end{itemize}
we can reach $(B,A)$ while at each step the pair of sets is a Radon partition with the same Radon point.  Being able to ``flip'' a Radon partition this way does not imply the existence of a Tverberg $3$-partition, as the following example shows.

Let $e_1, \dots, e_5$ be the canonical basis of $\rr^5$, and let $f_i = -e_i -e_{i+2}$ for $i=1,\dots, 5$, with the indices considered modulo $5$.  Let $S = \{e_1,\dots,e_5, f_1,\dots, f_5\}$ and $X_i = \{e_i, e_{i+2}, f_i\}$ for $i=1,\dots,5$.

Notice that $0 \in \conv X_i$ for all $i$.  Moreover, $X_i \cap X_{i+1}= \emptyset$ for all $i$.  We can go from the Radon partition $(X_1,X_2)$ by performing the steps
\begin{align*}
(X_1,X_2) \to (X_1 \cup X_3, X_2) \to (X_3, X_2) \to (X_3, X_2 \cup X_4) \to & (X_3,X_4) \to \\
 (X_3 \cup X_5, X_4) \to (X_5,X_4) \to (X_5, X_4\cup X_1) \to &(X_5,X_1) \to \\ (X_5 \cup X_2, X_1) \to &(X_2,X_1).    
\end{align*}

However, $S$ does not have any Tverberg $3$-partitions.  To show this, suppose $p \in T_3(S)$.  If $p$ has a non-zero coordinate, notice that there are at most two points of $S$ with the same sign on that coordinate, which contradicts the fact that $p$ is in a Tverberg $3$-partition.

If $p =0$, suppose that $A,B,C$ is a Tverberg $3$-partition witnessing $0$.  One of the three sets $A,B,C$ has at most one element of the form $e_i$.  Therefore, it cannot have any of the $f_j$, as it would need to have both $e_j$ and $e_{j+2}$ to cancel both coordinates.  This means that the set cannot have $0$ in its convex hull.

The reverse direction is straightforward: the existence of a Tverberg $3$-partition implies the possibility of flipping a Radon partition.  Understanding the gap between these conditions is an interesting problem.  Indeed, if $(A,B,C)$ is a Tverberg $3$-partition of $S$, then we can take the sequence of Radon partitions

\begin{align*}
    (A,B) \to (A, B \cup C) \to (A, C) \to (A \cup B, C) \to (B, C) \to (B, A\cup C) \to (B,A).
\end{align*}

\begin{problem}
    Determine conditions needed on a set $S \subset \rr^d$ so that $C^1_r(S) \neq \emptyset$ implies $T_{r+1}(S) \neq \emptyset$.
\end{problem}

More generally,

\begin{problem}
Under what additional assumptions can the non-emptiness of \(C_r^t(S)\) be upgraded to the existence of a Tverberg \((r+t)\)-partition of \(S\)?
\end{problem}

\subsection{Tverberg with tolerance}

The notion of $C^t_r(f)$ is similar to Tverberg partitions with tolerance \cites{Soberon:2018gn, GarciaColin:2017id}, in which we seek a partition $A_1,\dots, A_r$ of a set $S$ such that
\[
\bigcap_{j=1}^r \conv (A_j \setminus S') \neq \emptyset \qquad \mbox{for all }S' \in \binom{S}{t}.
\]
In other words, the partition remains a Tverberg partition after any $t$ points of $S$ are removed.  Note that in $C^t_r(f)$, the point of overlap is constant but the partition may change.  For Tverberg with tolerance, the partition remains constant but the point of overlap may depend on the deleted point.

Since bounding the dimension of Tverberg sets implies properties of $C^t_r(f)$, it might be possible that similar geometric properties on $T_r(f)$ have consequences for Tverberg partitions with tolerance.

Since our results work on the topological level, studying the topological versions of Tverberg's theorem with tolerance is also interesting.

\subsection{On the number of extra vertices}

The term \(t(k+1)\) in \cref{thm:main} is optimal in simple examples.  Let us
explain this in the linear case.  Suppose first that \(d=k=1\), and let
\[
S=\{x_1<x_2<\cdots <x_{n+1}\}\subset \rr
\]
be a set of \(n+1\) distinct points on the line.  Then \(T_r(S)=[x_r,x_{n-r+2}]\) whenever \(n+1\ge 2r-1\).  As long as this interval is non-degenerate,
\(\dim T_r(S)=1\).

Now consider the higher core \(C_r^t(S)\).  If we delete \(t\) points from the
left, the left endpoint of the remaining \(r\)-Tverberg interval can move as far
right as \(x_{r+t}\).  If we delete \(t\) points from the right, the right endpoint
can move as far left as \(x_{n-r-t+2}\).  Therefore
\(C_r^t(S) = [x_{r+t},x_{n-r-t+2}]\). 

This interval is non-empty if and only if
\(r+t\le n-r-t+2\),
or equivalently \(n\ge 2r+2t-2= 2(r-1)+2t\).  This matches the corresponding value of $n$ for \cref{thm:main}.

Another interesting case is when \(r=1\) and $f$ is linear.  As mentioned in \cref{sec:rado}, in this case,
\(C_1^t(S)\) is the set of points that remain in the convex hull of \(S\) after
deleting any \(t\) points.  Equivalently, these are points of Tukey depth at least
\(t+1\) with respect to \(S\).  Rado's centerpoint theorem says that there exist points of depth greater than $t$ as long as
\[
|S|\ge t(k+1)+1,
\]
which is precisely what \cref{thm:main} gives when \(r=1\).  Rado's centerpoint theorem is known to be optimal.

%\subsection{Removing requirements on $r$ for linear maps.}

%\cref{thm:main} requires $r$ to be a prime power.  This is not unexpected, as it is a necessary condition for the topological Tverberg theorem.  However, since Tverberg's theorem holds without conditions on the parameters, it is natural to ask if such conditions can be removed from \cref{thm:main} when $f$ is linear.

%In this subsection we show a weaker version that indicates that this should be the case.  The biggest change from the theorem below and \cref{thm:main} is replacing the condition $\dim T_r(S)\le k$ by $\dim (\conv (T_r(S))) \le k$.  The cascade conjecture remains open even for this case, so the result below is still providing new structural information.

%\begin{theorem}
 %   Let $r,d,k,t$ be positive integers.  Let $S$ be a set of $(r-1)(d+1)+t(k+1)+1$ points in $\rr^d$.  If $\dim (\conv ( T_r(S))) \le k$, then $C^t_r(S) \neq \emptyset$.
%\end{theorem}

%\begin{proof}
 %   For each set $A \in \binom{S}{t}$, let $K_A = \conv T_r(S \setminus A)$.  We have $K_A \subset \conv T_r(S)$, so all $K_A$ are contained in a common $k$-dimensional affine subspace of $\rr^d$.

  %  For any $k+1$ sets $K_{A_1},\dots, K_{A_{k+1}}$ we have \(S \setminus (A_1 \cup \dots \cup A_{k+1}) \ge (r-1)(d+1)+1\).  Therefore, we can apply Tverberg's theorem to deduce $T_r(S \setminus (A_1 \cup \dots \cup A_{k+1})) \neq \emptyset$.  We also have
   % \[
    %T_r(S \setminus (A_1 \cup \dots \cup A_{k+1})) \subset \bigcap_{j=1}^{n+1}
    %\]

%\end{proof}

\subsection{Acknowledgments.}  The author thanks Gil Kalai, Nikola Sadovek, and Pavle Blagojevi\'c for helpful discussions on this topic.  The development of this paper involved LLM-based tools in the brainstorming stage.  All arguments were checked by the author.


\end{document}